\documentclass[10pt,arxiv]{agn_article}
\usepackage[]{agn_basic}
\usepackage[]{agn_macros}
\usepackage{agn_tikz}
\usepackage[]{agn_bib}
\usepackage{tabularx}
\usepackage{cancel}
\usepackage{enumerate}
\renewcommand{\PSym}{\Sym^{\!+}}

\newcommand{\Wiso}{W_{\mathrm{iso}}}
\newcommand{\etss}{\mathrm{(E}\text{-}\mathrm{TSS)}}
\newcommand{\wetss}{\mathrm{(WE}\text{-}\mathrm{TSS)}}
\newcommand{\bep}{\mathrm{(BE^+)}}
\newcommand{\bicoax}{\mathrm{(bi}\text{-}\mathrm{coax)}}
\newcommand{\semi}{\mathrm{(semi)}}
\newcommand{\invert}{\mathrm{(invert)}}
\renewcommand{\Rp}{\R^+}

\tikzset{
    state/.style={
           rectangle,
           rounded corners,
           draw=black, very thick,
           minimum height=2em,
           inner sep=6pt,
           text centered,
           },
}
\begin{document}
\title{\vspace{-2cm}Do we need Truesdell's empirical inequalities? \\ On the coaxiality of stress and stretch.}
\date{\today}
\author{%
	Christian Thiel\thanks{
		Christian Thiel,  \ \ Lehrstuhl f\"{u}r Nichtlineare Analysis und Modellierung, Fakult\"{a}t f\"{u}r Mathematik, Universit\"{a}t Duisburg-Essen,  Thea-Leymann Str. 9, 45127 Essen, Germany, email: christian.thiel@uni-due.de}
		\quad and\quad Jendrik Voss\thanks{%
		Jendrik Voss,\quad Lehrstuhl f\"{u}r Nichtlineare Analysis und Modellierung, Fakult\"{a}t f\"{u}r Mathematik, Universit\"{a}t Duisburg-Essen, Thea-Leymann Str. 9, 45127 Essen, Germany; email: max.voss@.uni-due.de}
		\quad and\quad Robert J.\ Martin\thanks{%
		Robert J.\ Martin,\quad Lehrstuhl f\"{u}r Nichtlineare Analysis und Modellierung, Fakult\"{a}t f\"{u}r Mathematik, Universit\"{a}t Duisburg-Essen, Thea-Leymann Str. 9, 45127 Essen, Germany; email: robert.martin@uni-due.de}
		\quad and\quad Patrizio Neff\thanks{%
		Corresponding author: Patrizio Neff,\quad Head of Lehrstuhl f\"{u}r Nichtlineare Analysis und Modellierung, Fakult\"{a}t f\"{u}r	Mathematik, Universit\"{a}t Duisburg-Essen, Thea-Leymann Str. 9, 45127 Essen, Germany, email: patrizio.neff@uni-due.de}
}
\maketitle
\vspace{-1em}
\begin{abstract}
	Truesdell's empirical inequalities are considered essential in various fields of nonlinear elasticity. However, they are often used merely as a sufficient criterion for semi-invertibility of the isotropic stress strain-relation, even though weaker and much less restricting constitutive requirements like the strict Baker-Ericksen inequalities are available for this purpose. We elaborate the relations between such constitutive conditions, including a weakened version of the empirical inequalities, and their connection to bi-coaxiality and related matrix properties. In particular, we discuss a number of issues arising from the seemingly ubiquitous use of the phrase \enquote{$X,Y$ have the same eigenvectors} when referring to commuting symmetric tensors $X,Y$.
\end{abstract}
\\[.65em]

\thispagestyle{empty}
\enlargethispage{4em}
{\parskip=0.5mm \tableofcontents}
\bigskip

%
%
%
%
\section{Introduction}

A basic principle in linear algebra states that two symmetric matrices $X,Y\in\Symn$ commute if and only if they have a common basis of eigenvectors, i.e.\ if and only if there exists a basis $b_1,\dotsc,b_n\in\R^n$ such that for each $i\in\{1,\dotsc,n\}$, $b_i$ is an eigenvector of both $X$ and $Y$. This property is used quite often in nonlinear elasticity theory, where certain pairs of stress and stretch or strain always commute in the isotropic case.

However, it is also quite common to encounter the claim that $X$ and $Y$ \enquote{have the same eigenvectors} if they commute \cite{ortiz2004variational,thompson2008perspectives,mihai2011positive}. While a charitable reading of this phrase (or similar statements \cite{bertram2008elasticity,vallee2008dual,basar2013nonlinear,agn_neff2014rediscovering}) allows for an interpretation in the above (correct) sense, one must be careful not to infer from it additional properties of the pair $X,Y$ which might not hold in general. For example, it is \emph{not} true that for all commuting $X,Y\in\Symn$, every eigenvector of $X$ is an eigenvector of $Y$ (or vice versa). In isotropic nonlinear elasticity, it is therefore important to carefully distinguish between the general case of commuting matrices and the properties of \emph{coaxiality} and \emph{bi-coaxiality}, the latter of which is closely connected to the notion of \emph{semi-invertibility} \cite{reiner1948elasticity,truesdell1952}, which was shown by Truesdell \cite{truesdell1975inequalities} to follow from the so-called \emph{empirical inequalities}.

\subsection{Truesdell's empirical inequalities}
After the second world war had ravaged Europe, it was Clifford Truesdell who started a great campaign \cite{truesdell1952,truesdell60,truesdell65} to revive the science of nonlinear solid mechanics and to put it on firm mathematical grounds. In the course of this endeavor, Truesdell was faced with what he called the \enquote{Hauptproblem der endlichen Elastizitätstheorie} \cite{truesdell1956}, that is to find reasonable constitutive restrictions on the stress response for the nonlinear modeling of rubber.

In isotropic nonlinear hyperelasticity, every energy potential $W\col F\mapsto W(F)$ can be expressed as a function of the three principal invariants
\begin{equation*}
	I_1=\tr(B)=\norm{F}^2\,,\;\;\; I_2=\frac{1}{2}[(\tr\,B)^2-\,\tr(B^2)]=\tr(\Cof B)=\norm{\Cof F}^2\,,\;\;\; I_3=\det(B)=(\det F)^2
\end{equation*}
of the left Cauchy-Green tensor $B=FF^T$ corresponding to the deformation gradient $F$. In the incompressible case $I_3=1=\det F$, this representation can be further simplified to $W(F)=\widetilde W(I_1,I_2)$. In a 1952 article \cite{truesdell1952}, Truesdell postulated the conditions
\begin{equation}
\label{eq:truesdellFirst}
	\frac{\partial \widetilde W}{\partial I_1}(I_1,I_2)>0\,,\qquad\frac{\partial \widetilde W}{\partial I_2}(I_1,I_2)\geq 0\,,
\end{equation}
which are reasonable in that they require the energy to increase monotonically with increasing difference of average lengths and areas to those in the reference state, respectively.\footnote{For a more detailed interpretation, see Remark \ref{remark:weakEmpiricalInterpretation} and consider eq.\ \eqref{eq:averageLengthConstrainedMin} for the special case $d=1$.} Experimental evidence for rubber did not seem to contradict these conditions.

Later, Truesdell became rather passionate about abandoning the hyperelastic framework, instead focusing on the more general representation \cite{richter1948,rivlinEricksen1997relations}
\begin{align}
	\widehat\sigma(B) = \beta_0\.\id + \beta_1\.B + \beta_{-1}\.B^{-1}
\end{align}
for the stress response $FF^T = B \mapsto \sigmahat(B)=\sigma$ in isotropic nonlinear Cauchy elasticity; here, $\beta_i$ are scalar-valued functions depending on the matrix invariants of $B$ and $\sigma$ is the Cauchy stress tensor. In the hyperelastic case, the coefficient functions $\beta_i$ are given by
\begin{equation}
	\beta_0=\frac{2}{\sqrt{I_3}}\left(I_2\.\frac{\partial W}{\partial I_2}+I_3\.\frac{\partial W}{\partial I_3}\right)\,,\qquad\beta_1=\frac{2}{\sqrt{I_3}}\.\frac{\partial W}{\partial I_1}\,,\qquad\beta_{-1}=-2\sqrt{I_3}\.\frac{\partial W}{\partial I_2}\,,
\end{equation}
while under the constraint of incompressibility, the hydrostatic pressure (given by $\beta_0$) remains undetermined and requirement \eqref{eq:truesdellFirst} reduces to
\begin{align}
	\beta_1>0\,,\qquad\beta_{-1}\leq 0\,.\label{eq:truesdellSecond}
\end{align}
These are the so-called (incompressible) empirical inequalities \cite{truesdell65,truesdell1963static,beatty2001seven,aron2000some}. For the general compressible case, Truesdell strengthened the empirical inequalities to
\begin{align}
	\beta_{-1}\leq 0\,,\quad\beta_0\leq 0\,,\quad\beta_1 > 0\,,\label{eq:truesdellThird}
\end{align}
from which he deduced the (nowadays obsolete \cite{rivlin1997principles}) ordered-force inequality \cite{truesdell65} as well as the Baker-Ericksen inequalities 
\cite{bakerEri54}.

Truesdell's Hauptproblem remains unsolved to this day, although major advances have been made in the hyperelastic framework, especially with the seminal introduction of the notion of polyconvexity by Sir John Ball in 1977 \cite{ball1976convexity,ball1977constitutive} and the weaker requirement of rank-one convexity \cite{ball1976convexity}.

In various papers \cite{destrade2012,batra1976deformation,mihai2013numerical}, Truesdell's empirical inequalities are used to ensure the semi-invertibility of the isotropic stress-strain relation $B\mapsto\sigmahat(B)$, i.e.\ the representability $B = \psi_0(B)\.\id + \psi_1(B)\.\sigmahat + \psi_2(B)\.\sigmahat^2$ (see Definition \ref{definition:semiInvertibility}).\footnote{Other uses of the empirical inequalities can be found, for example, in \cite{liu2012note} or \cite{pucci2015determination}, cf.\ \cite{agn_voss2018more}.} For example, Destrade et al.\ \cite{destrade2012} utilize the semi-invertibility to show that the deformation corresponding to a simple Cauchy shear stress is not a simple shear. Similarly, Batra \cite[p.\,110]{batra1976deformation} shows that for a law of elasticity that satisfies the empirical inequalities, the deformation corresponding to a simple tensile load $\sigma$ must be a simple extension.

However, the empirical inequalities, also named \textbf{e}mpirical \textbf{t}rue \textbf{s}tress \textbf{s}tretch $\etss$, are by no means necessary to obtain these results; in fact, they are too strict to include commonly employed hyperelastic energy functions and even lead to physically unreasonable material behavior \cite{dunn1984elastic,saravanan2011adequacy}. On the other hand, semi-invertibility is equivalent to the bi-coaxility of the mapping $B\mapsto\sigmahat(B)$  \cite{blume1994elastic,blume1992form}, which is, for example, implied by the (considerably weaker) strict Baker-Ericksen inequalities \cite{bakerEri54}.

\subsection{Overview}
In the following, we will discuss the connection between commuting matrices, coaxiality, bi-coaxiality and semi-invertibility as well as sufficient criteria for these properties in isotropic nonlinear elasticity, including the strict Baker-Ericksen inequalities and the empirical inequalities. Similar considerations can be found in an earlier article by Dunn \cite{dunn1984elastic}.

In order to provide a more accessible sufficient criterion for semi-invertibility, we also introduce the weak empirical inequalities $\wetss$ which are stricter than $\bep$ but not as restrictive as $\etss$ and are, in fact, satisfied by a large number of classical elastic energy functions. As an example, we show that $\wetss$ are fulfilled by the classical quadratic Hencky energy, which does not satisfy the (full) empirical inequalities.

In order for novel results in nonlinear elasticity to be applicable to the largest possible number of material models, one should make the least restricting assumptions which are still sufficient for the result to hold. Our aim is to provide a variety of conditions which are strong enough to allow for common deductions from the stress tensor to the form of the stretch tensor (which is often required), but weak enough to be satisfied by a large variety of interesting constitutive models.

%
%
%
\section{Isotropy and Coaxiality}
\label{section:basicLinearAlgebra}
In the following, we will reiterate some very basic, well-known results from linear algebra in order to highlight the exact differences between the notions of simultaneous diagonalizability, coaxiality and bi-coaxiality. The interconnections between these properties are visualized in Figure \ref{fig:matrixRelation}.
\begin{proposition}\label{proposition:commutingEqualsSimultan}
	Let $A\,,B\in\Sym(n)$. Then the following are equivalent:
	\begin{itemize}
		\item $A$ and $B$ \emph{commute}, i.e.\ $A\.B=B\.A$.
		\item $A$ and $B$ are \emph{simultaneously (orthogonally) diagonalizable}, i.e.\ there exists $Q\in \On$ with
	\[
		A=Q^T\diag(a_1,\cdots,a_n)\.Q\qquad\text{and}\qquad B=Q^T\diag(b_1,\cdots,b_n)\.Q\,.
	\]
	\end{itemize}
\end{proposition}
\begin{proof}
	See, for example, \cite[Theorem 1.3.12]{horn1990matrix}.
\end{proof}
\begin{definition}\label{def:coaxial}
	Let $A,B\in\Sym(n)$. Then $A$ is called \emph{coaxial to} $B$ if each eigenvector of $A$ is an eigenvector of $B$. 
\end{definition}

Note carefully that coaxiality is \emph{not} a symmetric relation; for example, let $A=\id=\diag(1,1)$ and $B=\diag(1,0)$. Then every eigenvector of $B$ is an eigenvector of $A$, while $(1,1)^T$ is an eigenvector of $A$ but not of $B$, which implies that although $A$ is coaxial to $B$, $B$ is not coaxial to $A$.

\begin{definition}
	Let $A,B\in\Sym(n)$. Then $A$ and $B$ are called \emph{bi-coaxial} if $A$ is coaxial to $B$ and $B$ is coaxial to $A$.
\end{definition}

Note also that two symmetric matrices $A,B\in\Symn$ can commute even if neither $A$ is coaxial to $B$ nor $B$ is coaxial to $A$; for example, consider the case $A=\diag(1,1,0)$ and $B=\diag(0,1,1)$.
\begin{lemma}\label{lemma:coaxialToCommuting}
	Let $A,B\in\Sym(n)$ such that $A$ is coaxial to $B$. Then $A$ and $B$ commute.
\end{lemma}
\begin{proof}
	$A$ is coaxial to $B$ if and only if every orthogonal eigenvector basis $Q\in \On$ of $A$ is an eigenvector basis of $B$, i.e.
	\[
		A=Q^T\diag(a_1,\cdots,a_n)\.Q\qquad\implies\qquad B=Q^T\diag(b_1,\cdots,b_n)\.Q\,.\qedhere
	\]
\end{proof}
The exact distinction between coaxiality and simultaneous diagonalizability (cf.\ Proposition \ref{proposition:commutingEqualsSimultan}) is given by the following lemma.
\begin{lemma}\label{lemma:commutingToCoaxial}
	Let $A\,,B\in\Sym(n)$ be commuting matrices with
	\[
		A=Q^T\diag(a_1,\cdots,a_n)\.Q\qquad\text{and}\qquad B=Q^T\diag(b_1,\cdots,b_n)\.Q\,.
	\]
	for $Q\in \On$. Then $A$ is coaxial to $B$ if and only if
	\begin{equation}
	\label{eq:commutingToCoaxialImplication}
		a_i=a_j\qquad\implies\qquad b_i=b_j\qquad\text{for all }\; i,j\in\{1,\cdots,n\}\,.
	\end{equation}
\end{lemma}
\begin{proof}
	Let $A$ be coaxial to $B$. Since the implication \eqref{eq:commutingToCoaxialImplication} is trivial for $i=j$, consider the sum $v=e_i+e_j$ of two canonical basis vectors with $i\neq j$ and $a\colonequals a_i=a_j$. We compute
	\[
		Q^T\diag(a_1,\cdots,a_n)\.Q\.Q^Tv=Q^T\diag(a_1,\cdots,a_n)\.v=Q^T(a\.v)=a\.Q^Tv\,,
	\]
	thus $Q^Tv$ is an eigenvector of $A$ to the eigenvalue $a$. Due to the assumed coaxiality, $Q^Tv$ is an eigenvector of $B$, i.e.\ there exists $b\in\R$ with
	\[
		b\.Q^Tv = B\.b\.Q^Tv = Q^T\diag(b_1,\cdots,b_n)\.Q\.Q^Tv = Q^T\diag(b_1,\cdots,b_n)v\,,
	\]
	i.e.\ $\diag(b_1,\cdots,b_n)\.v=b\.v$. Since $v=e_i+e_j$, this equality immediately yields $b_i=b_j=b$.
	
	Now, let implication \eqref{eq:commutingToCoaxialImplication} hold for all $i,j\in\{1,\cdots,n\}$. Consider the orthogonal basis $Q\in \On$ of eigenvectors of $A$ and an arbitrary eigenvector $v\in\R^n$ of $A$ to the eigenvalue $a$, i.e.\ $A\.v=a\.v$. Then
	\[
		\diag(a_1,\cdots,a_n)\.Q\.v=Q\.A\.v=a\.Q\.v\,,
	\]
	thus $v'\colonequals Q\.v$ is an eigenvector of $\diag(a_1,\cdots,a_n)$ to the eigenvalue $a$. In order to show that $v'$ is an eigenvector of $\diag(b_1,\cdots,b_n)$, we only need to establish that $b_i=b_j$ if $v'_i\neq0$ and $v'_j\neq0$.
	The latter implies that $a_i\.v'_i=a\.v'_i$ and $a_j\.v'_j=a\.v'_j$, since $v'$ is an eigenvector of $\diag(a_1,\cdots,a_n)$ to the eigenvalue $a$, hence $a_i=a=a_j$ and thus $b_i=b_j$ due to condition \eqref{eq:commutingToCoaxialImplication}. Therefore, $v'=Q\.v$ is an eigenvector of $\diag(b_1,\cdots,b_n)$ to some eigenvalue $b$ and thus
	\[
		B\.v=Q^T\diag(b_1,\cdots,b_n)\.Q\.v=Q^T(b\.Q\.v)=b\.v\,.\qedhere
	\]
\end{proof}
\begin{corollary}\label{corollary:distinct}
	If $A\in\Symn$ has $n$ distinct (i.e.\ only simple) eigenvalues, then $A$ is coaxial to every $B\in\Symn$ which commutes with $A$. In particular, if both $A,B\in\Symn$ have only simple eigenvalues and $AB=BA$, then $A$ and $B$ are bi-coaxial.
\end{corollary}
\begin{corollary}
\label{cor:correspondingEigenvalues}
	Let $a_1,\dotsc,a_n\in\R$ denote the eigenvalues (with multiplicity) of $A\in\Symn$. If $A$ is coaxial to $B\in\Symn$, then there exist uniquely determined $b_1,\dotsc,b_n\in\R$ such that each eigenvector of $A$ to $a_i$ is an eigenvector of $B$ to $b_i$. We call $b_1,\dotsc,b_n$ the \emph{corresponding} eigenvalues to $a_1,\dotsc,a_n$.
\end{corollary}

Since the values $a_i$ are uniquely determined by $A$ up to ordering, the \enquote{correspondence} of the eigenvalues $b_i$ of $B$ can simply be interpreted as a compatibility requirement on the ordering of the values $b_i$. Note that if $A$ and $B$ are simultaneously diagonalized, then the ordering of the diagonal entries are consistent with this correspondence.

\subsection{Isotropic tensor functions}
In nonlinear elasticity theory, the stress-strain relation is often assumed to be an \emph{isotropic} mapping from a stretch or strain tensor to a (work conjugate) stress tensor. In the following, we will discuss the relation between isotropic mappings and (bi-)coaxiality.

\begin{definition}
	Let $\Phi\col M\to\Symn$ be a tensor function on an isotropic set\footnote{%
		A set $M\subset\Symn$ is called isotropic if $Q^TX\.Q\in M$ for all $X\in M$ and $Q\in\On$.%
	}
	$M\subset\Symn$. Then
	\begin{itemize}
		\item[i)]
			$\Phi$ is called \emph{isotropic} if $\Phi(Q^TX\.Q)=Q^T\Phi(X)\.Q$ for all $X\in M$ and $Q\in \On$,
		\item[ii)]
			$\Phi$ is called \emph{coaxial} if $X$ is coaxial to $\Phi(X)$ for all $X\in M$,
		\item[iii)]
			$\Phi$ is called \emph{bi-coaxial} if $X$ and $\Phi(X)$ are bi-coaxial for all $X\in M$.
	\end{itemize}
\end{definition}
\begin{lemma}[{\cite[Theorem 4.2.4]{ogden1997non}}]\label{lemma:isotropToCoaxial}
	Every isotropic tensor function $\Phi\col M\to\Sym(n)$ is coaxial.
\end{lemma}
\begin{proof}
	Let $v$ be an eigenvector of $X$ and $Q\in \On$ be the reflection at the hyperplane orthogonal to $v$, i.e.
	\[
		Q\.v=-v\qquad\text{and}\qquad Q\.x=x\qquad\text{for all}\quad x\in\R^n\quad \text{with}\quad \iprod{v,x}=0\,.
	\]
	Then $QXQ^T=X$, and due to the isotropy of $\Phi$,
	\[
		Q\Phi(X)Q^T = \Phi(QXQ^T) = \Phi(X)\qquad\implies\qquad Q\.\Phi(X) = \Phi(X)\.Q
	\]
	and thus
	\[
		Q\Phi(X)\.v\ =\ \Phi(X)Q\.v\ =\ -\Phi(X)\.v\,.
	\]
	Therefore, due to the definition of $Q$, there exists $\lambda\in\R$ with $\Phi(X)\.v=\lambda\.v$, i.e.\ $v$ is eigenvector of $\Phi(X)$.
\end{proof}

On the other hand, not every coaxial tensor function is isotropic. For example, consider the mapping $\Phi\col\Sym(2)\to\Sym(2)$ with $\Phi(X) = X_{11}\.\id$. Then for every $X\in\Symn(2)$, each $v\in\Rnn$ is an eigenvector of the identity $\Phi(X)$, which implies that $\Phi$ is coaxial. However, $\Phi$ is not isotropic, since
\[
	\Phi\Biggl(\underbrace{\matr{0&1\\1&0}}_{Q^T}\underbrace{\matr{1&0\\0&2}}_{X}\underbrace{\matr{0&1\\1&0}}_{Q}\Biggr)=\Phi\matr{2&0\\0&1}=\matr{2&0\\0&2}\neq\matr{1&0\\0&1}=\matr{0&1\\1&0}\Phi\biggl(\matr{1&0\\0&2}\biggr)\matr{0&1\\1&0}\,.
\]

The following Lemma characterizes the crucial difference between coaxiality and bi-coaxiality of an isotropic tensor function.
\begin{lemma}\label{lemma:isotropToBicoaxial}
	Let $\Phi\col M\to\Sym(n)$ be an isotropic tensor function. Then $\Phi$ is bi-coaxial if and only if for every $X\in M$,
	\[
		a_i\neq a_j\qquad\implies\qquad b_i\neq b_j\qquad\text{for all }\; i,j\in\{1,\cdots,n\}\,,
	\]
	where $a_i$ are the eigenvalues of $X$ and $b_i$ the \emph{corresponding} eigenvalues (cf.\ Corollary \ref{cor:correspondingEigenvalues}) of $\Phi(X)$.
\end{lemma}
\begin{proof}
	Due to Lemma \ref{lemma:isotropToCoaxial}, the argument $X$ is coaxial to $\Phi(X)$, thus according to Lemma \ref{lemma:coaxialToCommuting}, $X$ and $\Phi(X)$ commute. Then Lemma \ref{lemma:commutingToCoaxial} states that $\Phi(X)$ is coaxial to $X$ if and only if
	\[
		b_i=b_j\quad\implies\quad a_i=a_j\,,\qquad\text{i.e.}\qquad a_i\neq a_j\quad\implies\quad b_i\neq b_j\qquad\text{for all }\; i,j\in\{1,\cdots,n\}\,.\qedhere
	\]
\end{proof}
\begin{figure}[h!]
	\centering
	\tikzremake
	\begin{tikzpicture}[->,>=stealth',implies/.style={thick,double,double equal sign distance,-implies},iff/.style={thick,double,double equal sign distance,implies-implies}]
	 		\node[state] (1){
	 			$A\mapsto B$ isotropic
	 		};
		 	\node[state, right of=1, node distance=8cm] (2){
	  			$A,B$ bi-coaxial
	 		};
	 		\node[state, below of=1, node distance=2cm,xshift= 4cm] (3){%
	  			$A$ coaxial to $B$
	 		};
	 		\node[state, below of=2, node distance=4cm, align=center] (4){
				$A,B$ simultaneously \\ diagonalizable					  			
	 		};
	 		\node[state, left of=4,node distance =8cm](5){
	 			$A,B$ commute
	 		};
	 		\path[line width=1mm]
	 			(1) edge[implies] node[midway, below left]{Lemma \ref{lemma:isotropToCoaxial}} (3)
	 			(2) edge[implies] (3)
	 			(3) edge[implies] node[midway, below left]{Lemma \ref{lemma:coaxialToCommuting}} (4)
	 			(4) edge[iff] node[midway, below]{Proposition \ref{proposition:commutingEqualsSimultan}} (5)
	 			(1) edge node[midway,above]{Lemma \ref{lemma:isotropToBicoaxial}}  node[anchor=midway,below]{$a_i\neq a_j\implies b_i\neq b_j$} (2)
	 			(4) edge [bend right=35] node[pos=.35,above right,align=center]{Lemma \ref{lemma:commutingToCoaxial}\\$a_i= a_j\implies b_i= b_j$} (3.east)
	 		;
	\end{tikzpicture}
	\caption{Relational properties of two matrices $A\,,B\in\Sym(n)$ with corresponding eigenvalues $a_i$, $b_i$.}
	\label{fig:matrixRelation}
\end{figure}
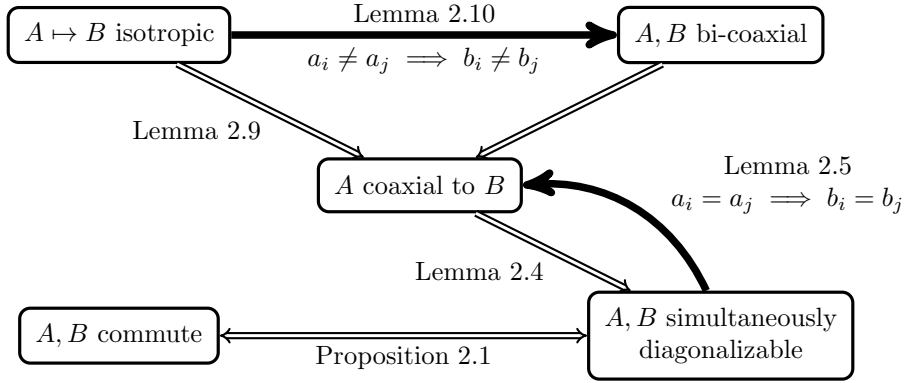

\subsection{Representation of isotropic tensor functions}
The classical representation of isotropic functions using coefficient functions of the matrix invariants is well known and widely utilized in continuum mechanics. Here, we discuss the exact conditions of this representability in detail, highlighting again the exact role played by coaxiality and isotropy in order to emphasize the additional information provided by the bi-coaxial case later on.

\begin{lemma}\label{lemma:repCoaxial}
	Let $A\,,B\in\Sym(n)$ such that $A$ is coaxial to $B$. Then there exist $\gamma_0,\cdots,\gamma_{n-1}\in\R$ with $B=\sum_{k=0}^{n-1}\gamma_k\.A^k$.
\end{lemma}
\begin{proof}
	Since $A$ is coaxial to $B$, due to Lemma \ref{lemma:coaxialToCommuting} there exists $Q\in \On$ with
	\[
		A=Q^T\diag(a_1,\cdots,a_n)\.Q\qquad\text{and}\qquad B=Q^T\diag(b_1,\cdots,b_n)\.Q\,.
	\]
	We only consider the distinct eigenvalues $a_{i_1},\cdots,a_{i_m}$ with $m\leq n$ and define the Vandermonde matrix
	\[
		V(a_{i_1},\ldots,a_{i_m}) \colonequals \begin{pmatrix} 1 & a_{i_1} & a_{i_1}^2 & \ldots & a_{i_1}^{m-1} \\  1 & a_{i_2} & a_{i_2}^2 & \ldots & a_{i_2}^{m-1} \\ \vdots & \vdots & \vdots &  & \vdots \\  1 & a_{i_m} & a_{i_m}^2 & \ldots & a_{i_m}^{m-1} \end{pmatrix}\,,\qquad\det V(a_{i_1},\ldots,a_{i_m})\ =\ \prod_{j<k}(a_{i_j}-a_{i_k})\neq 0\,.
	\]
	Here, the distinctness of the eigenvalues $a_{i_1},\cdots,a_{i_m}$ ensures that the determinant of the Vandermonde matrix is not equal to zero and thus that the system of equations
	\[
		b_{i_k} = \gamma_0 + \gamma_1a_{i_k} + \gamma_2a_{i_k}^2 + \ldots + \gamma_{m-1}a_{i_k}^{m-1}\qquad\text{for all }\; k\in\{1,\ldots,m\}
	\]
	has a single unique solution $(\gamma_0,\ldots,\gamma_{m-1})$. According to Lemma \ref{lemma:commutingToCoaxial}, $a_i=a_{i_k}$ imples $b_i=b_{i_k}$, thus the system of equations
	\[
		b_i = \gamma_0 + \gamma_1a_i + \gamma_2a_i^2 + \ldots + \gamma_{m-1}a_i^{m-1}\qquad\text{for all }\; i\in\{1,\ldots,n\}
	\]
	has a solution $(\gamma_0,\ldots,\gamma_{n-1})$. By simply expressing these equations in terms of diagonal matrices, we find
	\[
		\diag(b_1,\ldots,b_n) = \sum_{k=0}^{n-1}\gamma_k\diag(a_1,\ldots,a_n)^k\quad\iff\qquad B = \sum_{k=0}^{n-1}\gamma_k A^k\,.\qedhere
	\]
\end{proof}
\begin{proposition}\label{prop:repIsotrop}
Let $M\subset\Symn$ be an isotropic set.
	\begin{enumerate}[i)]
		\item
			A tensor function $\Phi\col M\to\Sym(n)$ is coaxial if and only if there exists a representation
			\begin{align}
				\Phi(X) = \sum_{i=0}^{n-1}\gamma_i\.X^i = \gamma_0\.\id + \gamma_1\.X + \ldots + \gamma_{n-1}\.X^{n-1}\qquad\text{for all }X\in M\label{eq:repSymCoaxialOnly}
			\end{align}
			where $\gamma_i$ are scalar-valued functions depending on $X$.
		\item
			A tensor function $\Phi\col M\to\Sym(n)$ is isotropic if and only if there exists a representation
			\begin{align}
				\Phi(X) = \sum_{i=0}^{n-1}\alpha_i\.X^i = \alpha_0\.\id + \alpha_1\.X + \ldots + \alpha_{n-1}\.X^{n-1}\qquad\text{for all }X\in M\label{eq:repSym}
			\end{align}
			where $\alpha_i$ are scalar-valued functions \emph{depending on the matrix invariants} of $X$.
		\item
			If $M\subset\PSymn$, then a tensor function $\Phi\col M\to\Sym(n)$ is isotropic if and only if there exists a representation
			\begin{align}
				\Phi(X) = \sum_{i=-1}^{n-2}\beta_i\.X^i = \beta_{-1}\. X^{-1} + \beta_0\. \id+ \ldots + \beta_{n-2}\.X^{n-2}\qquad\text{for all } X\in M\label{eq:repPSym}
			\end{align}
			where $\beta_i$ are scalar-valued functions depending on the matrix invariants of $X$.
	\end{enumerate}
\end{proposition}
\begin{proof}
	It is easy to verify that the given expressions are coaxial and isotropic, respectively. The representability of a coaxial mapping in the form \eqref{eq:repSymCoaxialOnly} follows directly from Lemma \ref{lemma:repCoaxial}.
	
	Now let $\Phi$ by isotropic. For given matrix invariants $I_1,\dotsc,I_n$, define $\alpha_i(I_1,\dotsc,I_n) = \gamma_i(\Xtilde)$ as the coefficients corresponding to the diagonal matrix $\Xtilde=\diag(a_1,\dotsc,a_n)$ with eigenvalues $a_1,\dotsc,a_n$ such that $I_k(\Xtilde)=I_k$. Then for all $X\in\Symn$ with invariants $I_1,\dotsc,I_n$, there exists $Q\in\On$ such that $X=Q^T\Xtilde Q$ and thus, due to the isotropy of $\Phi$,
	\begin{align*}
		\Phi(X) = \Phi(Q^T\Xtilde Q) &= Q^T\Phi(\Xtilde)Q\\
		&= Q^T\big[\gamma_0(\Xtilde)\.\id  + \ldots + \gamma_{n-1}(\Xtilde)\.(\Xtilde)^{n-1}\big]\,Q\\
		&= \gamma_0(\Xtilde)\.\id  + \ldots + \gamma_{n-1}(\Xtilde)\.(Q^T\Xtilde Q)^{n-1}\\
		&= \gamma_0(\Xtilde)\.\id  + \ldots + \gamma_{n-1}(\Xtilde)\.X^{n-1}
		\;=\; \alpha_0(I_1,\dotsc,I_n)\.\id  + \ldots + \alpha_{n-1}(I_1,\dotsc,I_n)\.X^{n-1}\,.
	\end{align*}
	
	Furthermore, for any $X\in\PSym(n)$, due to the Cayley-Hamilton Theorem there exist (invariant) coefficients $\betatilde_{-1},\betatilde_0,\cdots,\betatilde_{n-2}$ with
	\[
		0 = -X^n + \betatilde_{n-2}\.X^{n-1} + \ldots + \betatilde_0\.X + \betatilde_{-1}\.\id\qquad\iff\qquad X^{n-1} = \betatilde_{n-2}\.X^{n-2} + \ldots + \betatilde_0\.\id + \betatilde_{-1}\.X^{-1}\,.
	\]
	Combining this with equation \eqref{eq:repSym}, we find
	\[
		\Phi(X) = (\underbrace{\alpha_0 + \alpha_{n-1}\betatilde_0}_{\colonequals\beta_0})\.\id + (\underbrace{\alpha_1 + \alpha_{n-1}\betatilde_1}_{\colonequals\beta_1})\,X + \ldots + (\underbrace{\alpha_{n-2}+\alpha_{n-1}\betatilde_{n-1}}_{\colonequals\beta_{n-2}})\,X^{n-2} + \underbrace{\alpha_{n-1}\betatilde_{-1}}_{\colonequals\beta_{-1}}X^{-1}\,,
	\]
	which shows iii).
\end{proof}
%
%
%
\section{Cauchy stress and stretch}
One of the main areas of application for the above results is the theory of isotropic nonlinear elasticity, where the Cauchy stress response function $\sigmahat\col\PSym(3)\to\Sym(3),\, B\mapsto\sigmahat(B)=\sigma$, mapping the left Cauchy-Green stretch tensor $B=FF^T$ corresponding to a deformation gradient $F\in\GLpn$ to the Cauchy stress tensor $\sigma$, is isotropic and thus coaxial. Accordingly, the stress response can be represented in the form
\begin{align}
	\sigmahat(B) = \beta_0\.\id + \beta_1\.B + \beta_{-1}\.B^{-1}\,.\label{eq:repSigma}
\end{align}
In the following, we elaborate the connection between bi-coaxility and different constitutive requirements, including Truesdell's empirical inequalities. Although we will focus on the three-dimensional case, our results can easily be adapted to the general $n$-dimensional case, see \cite[p.~38]{agn_thiel2017neue}.

\begin{definition}
\label{definition:semiInvertibility}
	A mapping $\widehat\sigma\colon\PSym(3)\to\Sym(3)$ is called \emph{semi-invertible} if there exists a representation
	\begin{align}
		B = \psi_0\.\id + \psi_1\.\widehat\sigma + \psi_2\.\widehat\sigma^2\label{eq:semiInvertible}
	\end{align}
	with real valued functions $\psi_0$, $\psi_1$, $\psi_2$ depending on the principal invariants of $B$ (and not of $\sigma$).
\end{definition}
The semi-invertibility of the Cauchy stress response function is often used in isotropic nonlinear elasticity to infer particular properties of $B=FF^T$ from the form of a given Cauchy stress $\sigma=\sigmahat(B)$, cf.\ Section~\ref{section:mainSummary}. For example, Marzano \cite{marzano1983interpretation} used the representation formula \eqref{eq:semiInvertible} to show that any Cauchy stress tensor of the form $\sigma=\diag(s,0,0)$ with $s>0$, i.e.\ uniaxial tension, is caused by a simple stretch $V-\id=\sqrt{FF^T}-\id=\diag(\alpha,0,0)$ with $\alpha>0$ for any isotropic law of elasticity which satisfies the Baker-Ericksen inequalities,\footnote{%
	Marzano seems to claim that the reverse implication holds as well, i.e.\ that an elastic law for which a Cauchy stress of the form $\diag(s,0,0)$ with $s>0$ is only caused by a stretch $V$ with $V-\id=\sqrt{FF^T}-\id=\diag(\alpha,0,0)$, $\alpha>0$, always satisfies the Baker-Ericksen inequalities. However, this is not true in general, cf.\ Appendix \ref{appendix:marzano}.%
}
cf.\ Definition \ref{definition:bakerEricksen} and Lemma \ref{corollary:BEimpliesSemiInvertibility}.

Johnson and Hoger \cite{johnson1993dependence} have shown that if a mapping of the form \eqref{eq:repSigma} is semi-invertible, then the functions $\psi_i$ are given by
\begin{align}
	\psi_0 &= \frac{1}{\mathcal A}\left(\beta_0^2-2\.\beta_{-1}\beta_1+I_2\.\beta_1^2+I_3\.\frac{\beta_0\.\beta_1^2}{\beta_{-1}}+I_1\.\beta_{-1}^2+\frac{I_2}{I_3}\.\beta_0\.\beta_{-1}\right)\,,\notag\\
	\psi_2 &= -\frac{1}{\mathcal A}\left(2\.\beta_0+I_3\.\frac{\beta_1^2}{\beta_{-1}}+\frac{I_2}{I_3}\.\beta_{-1}\right)\,,\label{eq:psiSemiInvertible}\\
	 \psi_3 &= \frac{1}{\mathcal A}\,,\qquad\mathcal A = I_1\.\beta_1^2-I_3\.\frac{\beta_1^3}{\beta_{-1}}+\frac{1}{I_3}\beta_{-1}^2-\frac{I_2}{I_3}\.\beta_1\.\beta_{-1}\,.\notag
\end{align}
Note again that (by virtue of the functions $\beta_i$), the coefficient functions $\psi_i$ depend on the invariants of $B$, not on $\sigmahat(B)$. Therefore, equation \eqref{eq:semiInvertible} does not imply the invertibility\footnote{Lurie has noted that \enquote{[The inversion of $B\to\sigma(B)$] can only be solved completely for some especially "fortunate" prescriptions of the dependence of the function $W$ on the invariants} \cite[p.224]{lurie2012non}.} (i.e.\ the bijectivity) or even the injectivity of the mapping $B\mapsto\widehat\sigma(B)$.\footnote{%
	If, however, the functions $\psi_i$ are constant, then the mapping $\sigmahat$ is indeed invertible. In particular, this is the case if the functions $\beta_i$ are constant (which is, for example, satisfied for the Cauchy stress response function induced by the Mooney-Rivlin energy function $W=c_1\.(I_1-3)+c_2\.(I_2-3)$ in the incompressible framework). For additional remarks on the injectivity of the Cauchy stress response, see \cite{agn_neff2016injectivity,agn_mihai2016hyperelastic,agn_mihai2018hyperelastic}.%
}
For example, consider the mapping $\sigmahat\col\PSym(3)\to\Sym(3)$ with $\sigmahat(B)=\dev_3 B=B-\frac{1}{3}\tr(B)\.\id$, which is semi-invertible due to the equality $B=\frac{1}{3}\tr(B)\.\id+\sigmahat(B)$, but not invertible since $\sigmahat(\id)=0=\sigmahat(0)$. On the other hand, invertibility directly implies semi-invertibility.

\begin{lemma}
	Every invertible isotropic function $\widehat\sigma\colon\PSym(3)\to\Sym(3)$ is semi-invertible.
\end{lemma}
\begin{proof}
	The isotropy of $\sigmahat$ implies
	\[
		\sigmahat\inv(Q^TY\.Q)=\sigmahat\inv(Q^T\sigmahat(B)\.Q)=\sigmahat\inv(\sigmahat(Q^TB\.Q))=Q^TB\.Q=Q^T\sigmahat\inv(\sigmahat(B))\.Q=Q^T\sigmahat\inv(Y)\.Q
	\]
	for all $Q\in\OO(3)$ and all $Y=\sigmahat(B)\in\Sym(3)$, i.e.\ the isotropy of the inverse $\sigmahat\inv$. Then due to Proposition \ref{prop:repIsotrop}, $\sigmahat\inv$ can be written as
	\[
		\sigmahat\inv(Y) = \alpha_0\.\id + \alpha_1\.Y + \alpha_2\.Y^2\qquad\iff\qquad B = \alpha_0\.\id + \alpha_1\.\sigmahat(B) + \alpha_{2}\.\sigmahat(B)^{2}
	\]
	for all $Y=\sigmahat(B)\in\Sym(3)$ with coefficients
	$\alpha_i(I_1(\sigmahat(B)),I_2(\sigmahat(B)),I_3(\sigmahat(B)))$ depending on the matrix invariants of $Y=\sigmahat(B)$. The mappings $B\mapsto\alpha_i$ are isotropic, therefore there exists the representation $\alpha_i=\psi_i(I_1(B),I_2(B),I_3(B))$, yielding a representation of the form \eqref{eq:semiInvertible}.
\end{proof}

The following proposition shows the underlying connection between the concepts discussed in Section \ref{section:basicLinearAlgebra} and the notion of semi-invertibility.

\begin{proposition}
	An isotropic function $\widehat\sigma\colon\PSym(3)\to\Sym(3)$ is semi-invertible if and only if it is bi-coaxial.
\end{proposition}
\begin{proof}
	Let $\sigmahat$ be semi-invertible. Due to the isotropy, $B$ is coaxial to $\sigmahat(B)$ for all $B\in\PSym(3)$ according to Lemma \ref{lemma:isotropToCoaxial}, thus it remains to show that $\sigmahat(B)$ is coaxial to $B$. Let $v\in\R^3$ be an eigenvector of $\sigmahat(B)$ to the eigenvalue $\lambda\in\R$. Then the semi-invertibility of $\sigmahat$ implies
	\[
		B\.v=\psi_0\.\id\.v + \psi_1\.\widehat\sigma(B)\.v + \psi_2\.\widehat\sigma(B)^2\.v=\psi_0\.v + \psi_1\.\lambda\.v + \psi_2\.\lambda^2v=(\psi_0 + \psi_1\.\lambda + \psi_2\.\lambda^2)\.v\,,
	\]
	thus $v$ is an eigenvector of $B$ for each eigenvector $v$ of $\sigmahat(B)$, which implies that $\sigmahat$ is bi-coaxial.
	
	Now let $\sigmahat$ be a bi-coaxial mapping. Then in particular, $\sigmahat(B)$ is coaxial to $B$ for every $B\in\PSym(3)$. According to Lemma \ref{lemma:repCoaxial}, there exists a representation
	\begin{equation}\label{eq:semiInvertibilityFixedCoefficients}
		B = \alpha_0\.\id + \alpha_1\.\widehat\sigma(B) + \alpha_2\.\widehat\sigma(B)^2
	\end{equation}
	with $\alpha_i\in\R$ depending on $B$. Due to the isotropy of $\sigmahat$,
	\[
		Q^TB\.Q=Q^T\left[\alpha_0\.\id + \alpha_1\.\widehat\sigma(B) + \alpha_2\.\widehat\sigma(B)^2\right]Q=\alpha_0\.\id + \alpha_1\.\widehat\sigma(Q^TB\.Q) + \alpha_2\.\widehat\sigma(Q^TB\.Q)^2\,,
	\]
	i.e.\ the coefficients $\alpha_i$ are invariant under the transformation $B\mapsto Q^TBQ$. Thus there exists a representation $\alpha_i=\psi_i(I_1(B),I_2(B),I_3(B))$ which, combined with \eqref{eq:semiInvertibilityFixedCoefficients}, yields the semi-invertibility of $\sigmahat$.
\end{proof}
\subsection{Constitutive requirements in nonlinear elasticity}

In nonlinear elasticity theory, additional assumptions on the stress-strain relation are commonly known as \enquote{constitutive requirements}. A number of these requirements are directly connected to the semi-invertibility, i.e.\ the bi-coaxiality of the Cauchy stress response function.

Corollary \ref{cor:correspondingEigenvalues} allows us to employ the following notation here and throughout: For an isotropic function $\widehat\sigma\colon\PSym(3)\to\Sym(3)$ and given $B\in\PSym(3)$, the eigenvalues of $B$ will be denoted by $\lambda_1^2,\dotsc,\lambda_n^2$, and the corresponding eigenvalues of $\sigmahat(B)$ will be denoted by $\sigma_1,\dotsc,\sigma_n$. For a Cauchy stress response function $\sigmahat$ depending on the left Cauchy-Green stretch tensor $B$, the value $\sigma_i$ represents the principal stress corresponding to the principal stretch $\lambda_i$.

\begin{definition}
\label{definition:bakerEricksen}
	An isotropic function $\widehat\sigma\colon\PSym(3)\to\Sym(3)$ satisfies the so-called \emph{Baker-Ericksen inequalities} (BE) \cite{bakerEri54} if
	\begin{equation}\label{eq:BE}
		\lambda_i\geq \lambda_j\qquad\implies\qquad \sigma_i\geq\sigma_j\qquad\text{for all }\; i,j\in\{1,\cdots,n\} \tag{BE}
	\end{equation}
	for all $B\in\PSym(3)$. Furthermore, $\sigmahat(B)$ satisfies the so-called \emph{strict Baker-Ericksen inequalities} $\bep$ if
	\begin{equation}\label{eq:BEp}
		\lambda_i> \lambda_j\qquad\implies\qquad \sigma_i>\sigma_j\qquad\text{for all }\; i,j\in\{1,\cdots,n\} \tag{BE$^+$}
	\end{equation}
	for all $B\in\PSym(3)$.
\end{definition}
Note that $\lambda_i= \lambda_j$ implies $\sigma_i=\sigma_j$ due to Lemma \ref{lemma:commutingToCoaxial}, thus $\bep$ implies (BE). Furthermore, the strict Baker-Ericksen inequalites ensure that $\lambda_i\neq \lambda_j$ implies $\sigma_i\neq\sigma_j$ and therefore, with Lemma \ref{lemma:isotropToBicoaxial}, condition $\bep$ implies bi-coaxility and thus semi-invertibility of the stress-strain relation.
\begin{proposition}
\label{corollary:BEimpliesSemiInvertibility}
	Let $\widehat\sigma\colon\PSym(3)\to\Sym(3)$ be an isotropic tensor function which satisfies $\bep$. Then $\widehat\sigma$ is semi-invertible.
\end{proposition}
While invertibility and (BE$^+$) both ensure the semi-invertibility, they are not immediately related: The function $\sigmahat(B)=\id-B$, for example, is invertible but does not satisfy the Baker-Ericksen inequalities, while $\sigmahat(B)=\dev_3 B$ is not invertible but satisfies $\bep$, since
\[
	\sigma_i>\sigma_j\qquad\iff\qquad\lambda_i-\frac{1}{3}(\lambda_1+\lambda_2+\lambda_3)>\lambda_j-\frac{1}{3}(\lambda_1+\lambda_2+\lambda_3)\qquad\iff\qquad \lambda_i>\lambda_j
\]
for all $i,j\in\{1,\cdots,n\}$.
An even more restricting requirement ensuring the semi-invertibility is given by Truesdell's empirical inequalites \cite{truesdell1975inequalities,truesdell2012elements}.
\begin{definition}
	An isotropic function $\widehat\sigma\colon\PSym(3)\to\Sym(3)$ with
	\[
		\sigmahat(B) = \beta_0\.\id + \beta_1\.B + \beta_{-1}\.B^{-1}
	\]
	satisfies the \emph{empirical} inequalities or \emph{\textbf{e}mpirical \textbf{t}rue \textbf{s}tress \textbf{s}tretch} $\etss$ inequalities\footnote{Also called an \emph{adscititious inequality} by Truesdell \cite{moon1974interpretation}.} if
	\[
		\beta_{-1}\leq 0\,,\quad\beta_0\leq 0\,,\quad\beta_1 > 0 \quad\text{ for all }\; B\in\PSym(3)\,.
	\]
\end{definition}

\subsection{The weak empirical inequalities}
Due to \eqref{eq:psiSemiInvertible}, the empirical inequalities ensure $\psi_i>0$ for all $i\in\{1,2,3\}$. We will show that the condition $\beta_0\leq 0$ is not necessary for the bi-coaxiality of the elasticity law and is too strict for most hyperelastic energy functions. Therefore, we introduce in the following the weaker inequalities.
\begin{definition}
\label{definition:weakEmpirical}
	An isotropic function $\widehat\sigma\colon\PSym(3)\to\Sym(3)$ with
	\[
		\sigmahat(B) = \beta_0\.\id + \beta_1\.B + \beta_{-1}\.B^{-1}
	\]
	satisfies the \emph{\textbf{w}eak \textbf{e}mpirical \textbf{t}rue \textbf{s}tress \textbf{s}tretch} $\wetss$ or \emph{weak empirical inequalities} if for all $B\in\PSym(3)\setminus\Rp\cdot\id$,
	\begin{equation}\label{eq:definitionWeakEmpirical}
		\beta_{-1}\leq 0 \qquad\text{and}\qquad \beta_1 \geq 0
	\end{equation}
	with one of the two inequalities being strict.\footnote{%
		The requirement that for each $B$ one of the inequalities in \eqref{eq:definitionWeakEmpirical} is strict can be stated more formally as \enquote{for each $B\in\PSym(3)\setminus\Rp\cdot\id$, there exists $i\in\{-1,1\}$ such that $\beta_i\neq0$}, which is a weaker requirement than \enquote{there exists $i\in\{-1,1\}$ such that for all $B\in\PSym(3)\setminus\Rp\cdot\id$, \;$\beta_i\neq0$}.%
	}
\end{definition}
Note carefully that the inequalities \eqref{eq:definitionWeakEmpirical} are not required to hold if $B=FF^T$ is a multiple of the identity tensor $\id$, which corresponds to the case $F\in\Rp\cdot\SO(3)$ of a conformal deformation tensor or, equivalently, three equal principal stretches $\lambda_1=\lambda_2=\lambda_3$. Of course, if $\beta_1$ and $\beta_{-1}$ are continuous, than the (non-strict) inequalities are still satisfied in this case.

However, there are multiple reasons to exclude these purely volumetric stretch tensors from the above definition. First, the geometric interpretation given in Remark \ref{remark:weakEmpiricalInterpretation} is no longer directly applicable to purely volumetric stretches, thus there is no direct physical motivation in this case. Second, the case of non-simple eigenvalues is often notoriously difficult from a computational point of view (cf.\ the example of the Hencky energy discussed in Section \ref{section:henckyExample}), thus by omitting the case $\lambda_1=\lambda_2=\lambda_3$, the task of verifying that the criterion $\wetss$ holds for a given constitutive model may be simplified considerably. Furthermore, unless continuity of the functions $\beta_i$ is explicitly required, their values at $B=\lambda\.\id$ are not uniquely determined by the stress response alone; note that in this case, $\sigma(B)=(\beta_0+\lambda\beta_1+\lambda\inv\beta_{-1})\.\id$.

Finally, our main motivation for introducing an alternative to the classical empirical inequalities is to provide a weaker condition which is still sufficient to ensure the semi-invertibility of the stress response. The following lemma states that for this purpose, the requirements given in Definition \ref{definition:weakEmpirical} are still adequate.
\begin{lemma}
	Any isotropic function $\widehat\sigma\colon\PSym(3)\to\Sym(3)$ which satisfies the weak empirical inequalities (WE-TSS) satisfies the strict Baker-Ericksen inequalites.
\end{lemma}
\begin{proof}
	Due to the isotropy of $\sigmahat$, we can assume $B$ (and thus $\sigmahat(B)$) to be in diagonal form without loss of generality. Then the representation $\sigmahat(B) = \beta_0\.\id + \beta_1\.B + \beta_{-1}\.B^{-1}$ yields
	\begin{alignat*}{2}
		&&\diag(\sigma_1,\sigma_2,\sigma_3) &= \beta_0\.\id+\beta_1\diag(\lambda_1^2,\lambda_2^2,\lambda_3^2)+\beta_{-1}\diag(\lambda_1^2,\lambda_2^2,\lambda_3^2)^{-1}\\
		&\iff\quad& \sigma_k &= \beta_0+\beta_1\.\lambda_i+\beta_{-1}\.\lambda_k^{-1}\qquad\qquad\text{for all }\; k\in\{1,2,3\}\,.
	\end{alignat*}
	Now for arbitrary $i,j\in\{1,2,3\}$, if $\lambda_i>\lambda_j$, then the weak empirical inequalities $\beta_{-1}\leq 0$ and $\beta_1\geq 0$ ensure that $\beta_{-1}\.\lambda_i^{-2}\geq\beta_{-1}\.\lambda_j^{-2}$ and $\beta_1\.\lambda_i^2\geq\beta_1\.\lambda_j^2$ with one inequality being strict. Thus
	\[
		\sigma_i=\beta_0+\beta_1\.\lambda_i^2+\beta_{-1}\.\lambda_i^{-2}>\beta_0+\beta_1\.\lambda_j^2+\beta_{-1}\.\lambda_j^{-2}=\sigma_j\qquad\text{for all }\; i,j\in\{1,\cdots,n\}\,.\qedhere
	\]
\end{proof}

\subsection{Summary}
\label{section:mainSummary}
The above results can be summarized as follows (cf.\ \cite{dunn1984elastic,dunn1983certain}).
\begin{theorem}
	Let $\widehat\sigma\colon\PSym(3)\to\Sym(3)$ be an isotropic Cauchy stress response function with
	\begin{equation}
		\widehat\sigma(B) = \beta_0\.\id + \beta_1\.B + \beta_{-1}\.B^{-1}\qquad\text{for all }\; B\in\PSym(3)\,,
	\end{equation}
	where $\beta_{-1}$, $\beta_0$, $\beta_1$ are scalar-valued functions depending on the left Cauchy-Green tensor $B$. Let $\lambda_1^2$, $\lambda_2^2$, $\lambda_3^2$ and $\sigma_1$, $\sigma_2$, $\sigma_3$ denote the eigenvalues of $B$ and the corresponding eigenvalues of $\sigmahat(B)$, respectively.
	
	Then the following implications hold:
	
	\fbox{\parbox{.98\textwidth}
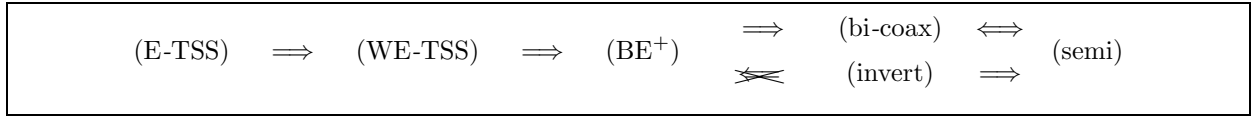
{
		\[
			\etss\quad\implies\quad\wetss\quad\implies\quad\bep\qquad\begin{matrix}\implies&\quad\bicoax\quad&\iff\\[.5em]\xcancel\Longleftarrow&\quad\invert\quad&\implies\end{matrix}\quad\semi
		\]
	}}
	
	\bigskip
	
	{\setlength{\tabcolsep}{1em}
	\begin{tabularx}{\textwidth}{rX}
		$\etss\mathrm{\,:}$ & $\widehat\sigma$ satisfies the \emph{empirical inequalities}, i.e.\\[.5em]
		\phantom{$\wetss\mathrm{\,:}$}& $\beta_{-1}\leq 0$, $\beta_0\leq 0$, $\beta_1 > 0$ for all $B$.\\[1.2em]
	\end{tabularx}
	\begin{tabularx}{\textwidth}{rX}
		$\wetss\mathrm{\,:}$ & $\widehat\sigma$ satisfies the \emph{weak empirical inequalities}, i.e.\\[.5em]
		\phantom{$\wetss\mathrm{\,:}$}& $\beta_{-1}\leq 0$, $\beta_1\geq 0$ for all $B\notin\Rp\cdot\id$, and for each $B$ one of the two inequalities is strict.\\[1.2em]
	\end{tabularx}
	\begin{tabularx}{\textwidth}{rX}
		$\bep\mathrm{\,:}$ & $\widehat\sigma$ satisfies the \emph{strict Baker-Ericksen inequalities}, i.e.\\[.5em]
		\phantom{$\wetss\mathrm{\,:}$}& $\lambda_i>\lambda_j\ \implies \sigma_i>\sigma_j$ for all $i,j\in\{1,2,3\}$.\\[1.2em]
	\end{tabularx}
	\begin{tabularx}{\textwidth}{rX}
		$\bicoax\mathrm{\,:}$ & $\widehat\sigma$ is \emph{bi-coaxial}, i.e.\\[.5em]
		\phantom{$\wetss\mathrm{\,:}$}& every eigenvector of $B$ is an eigenvector of $\widehat\sigma(B)$ and vice versa.\\[1.2em]
	\end{tabularx}
	\begin{tabularx}{\textwidth}{rX}
		$\semi\mathrm{\,:}$ & $\widehat\sigma$ is \emph{semi-invertible}, i.e.\\[.5em]
		\phantom{$\wetss\mathrm{\,:}$}& there exist real-valued functions $\psi_0,\psi_1,\psi_2$ depending on the matrix invariants of $B$ such that $B = \psi_0\.\id + \psi_1\.\widehat\sigma(B) + \psi_2\.\widehat\sigma(B)^2$ for all $B\in\PSym(3)$.\\[1.2em]
	\end{tabularx}
	\begin{tabularx}{\textwidth}{rX}
		$\invert\mathrm{\,:}$ & $\widehat\sigma$ is \emph{invertible}, i.e.\\[.5em]
		\phantom{$\wetss\mathrm{\,:}$}& there exists an inverse function $\widehat\sigma\inv\colon\Sym(3)\to\PSym(3)$ with $B=\widehat\sigma\inv\bigl(\widehat\sigma(B)\bigr)$ for all $X\in\PSym(3)$ and $Y = \widehat\sigma\bigl(\widehat\sigma^{-1}(Y)\bigr)$ for all $Y\in\Sym(3)$.
	\end{tabularx}
	}
\end{theorem}

For typical material models considered in isotropic nonlinear Cauchy elasticity, the particularly important Cauchy stress response function $B=FF^T\mapsto\sigmahat(B)$ is usually not an invertible mapping. However, it is often required to find the stretch or strain \emph{induced by} a given Cauchy stress $\sigma$ or, if not uniquely determined, to deduce the \emph{general form} of all $B\in\PSym(3)$ with $\sigmahat(B)=\sigma$. This can be considered one of the main purposes behind the notion of semi-invertibility: although it is often difficult to find an explicit representation of the coefficient functions $\beta_i$ (and thus of the functions $\psi_i$) for a given constitutive law of elasticity, the representation given by eq.\ \eqref{eq:semiInvertible} still provides a certain amount of information about $B$ if $\sigma=\sigmahat(B)$ is known \cite{destrade2012,batra1976deformation,mihai2013numerical}.

In terms of isotropic nonlinear elasticity, the observation (cf.\ Corollary \ref{corollary:distinct}) that $\sigmahat(B)$ and $B$ are bi-coaxial if $\sigmahat(B)$ has only simple eigenvalues (but not in general) can be stated as follows: although each principal axis of (Eulerian) strain must be a principal axis of (Cauchy) stress, the reverse must not hold in general unless the principal stresses are pairwise distinct or additional constraints on the constitutive law are assumed to hold.\footnote{%
	In a recent contribution \cite{thiel2018shear}, this consequence of the distinctness of principal stresses is further discussed and utilized to verify and generalize a statement by Destrade et al.\ \cite{destrade2012} on the deformations corresponding to simple shear Cauchy stresses. Similar work can be found in an early paper by Moon and Truesdell \cite{moon1974interpretation} as well as in a more recent article by Mihai and Goriely \cite{mihai2011positive}.%
}
Among the constitutive requirements which ensure this bi-coaxiality of stress and strain, i.e.\ the semi-invertibility of $\sigmahat$, are the (weak) empirical inequalities, although the (weaker) strict Baker-Ericksen inequalities are sufficient as well.

Unfortunately, the condition of semi-invertibility is difficult to verify directly for a given isotropic function. Therefore, if this property is required to obtain particular results, it is common to assume a (stronger) requirement for the constitutive law to hold which implies semi-invertibility. This has led many authors to restrict their considerations to constitutive laws which satisfy Truesdell's empirical inequalities, thereby excluding a large number of elastic models which are widely used in applications throughout different fields of research. However, in many cases encountered in the literature \cite{destrade2012,batra1976deformation,mihai2013numerical}, the proof of a result relies only on the semi-invertibility (i.e.\ the bi-coaxiality) of the Cauchy stress response, not on any further properties obtained from the empirical inequalities, and could therefore be formulated in a much more general way (for example, under the condition of the bi-coaxiality of the Cauchy stress response, the strict Baker-Ericksen inequalities or the weak empirical inequalities introduced here).

%
%
%
\section{Hyperelasticity}
The basic assumption of the hyperelastic framework is that the elastic stress response is induced by an \emph{energy potential} $W$, i.e.\ that there exists a function $W\col\GLp(3)\to\R$ such that the Cauchy stress $\sigma(F)$ corresponding to a deformation gradient $F\in\GLp(3)$ is given by $\sigma(F)=\frac{1}{\det F}\.DW(F)\.F^T$.

In the isotropic case (i.e.\ in the case that $B$ is coaxial to $\sigmahat(B)$, see \cite{vianello1996optimization}), the energy potential can be expressed as $W(F)=W(I_1,I_2,I_3)$ in terms of the principal invariants of $B=FF^T$. This representation allows for a direct computation of the coefficients $\beta_i$ \cite{mihai2017characterize}:
\begin{align}
	\beta_0=\frac{2}{\sqrt{I_3}}\left(I_2\.\frac{\partial W}{\partial I_2}+I_3\.\frac{\partial W}{\partial I_3}\right)\,,\qquad\beta_1=\frac{2}{\sqrt{I_3}}\.\frac{\partial W}{\partial I_1}\,,\qquad\beta_{-1}=-2\sqrt{I_3}\.\frac{\partial W}{\partial I_2}\,.\label{eq:CauchyHyper}
\end{align}
Since the invariants $I_1,I_2,I_3$ are positive for any $B\in\PSym(3)$, the weak empirical inequalities take a particularly simple form in this case.
\begin{proposition}\label{lemma:hyperWETSS}
	Let the Cauchy stress response function $\sigmahat$ be induced by the isotropic energy potential $W\col\GLp(3)\to\R$. Then the weak empirical inequalities $\wetss$ are satisfied if and only if for every $B\in\PSym(3)\setminus\Rp\cdot\id$,
	\begin{align}\label{eq:hyperWETSS}
		\frac{\partial W}{\partial I_1}(I_1,I_2,I_3)\geq 0\qquad\text{and}\qquad\frac{\partial W}{\partial I_2}(I_1,I_2,I_3)\geq 0
	\end{align}
	with one inequality being strict.
\end{proposition}
The additional condition $\beta_0\leq 0$ of the empirical inequalities $\etss$ is equivalent to
\begin{equation}
\label{eq:empiricalRedundantConditionHyper}
	0\geq\frac{2}{\sqrt{I_3}}\left(I_2\.\frac{\partial W}{\partial I_2}+I_3\.\frac{\partial W}{\partial I_3}\right)\qquad\iff\qquad \underbrace{-\frac{\partial W}{\partial I_2}}_{\leq 0}\geq\underbrace{\frac{I_3}{I_2}}_{\geq 0}\frac{\partial W}{\partial I_3}\,.
\end{equation}
Therefore, $\etss$ implies $\frac{\partial W}{\partial I_3}<0$ or $\frac{\partial W}{\partial I_3}=0$, the latter being satisfied if the energy function is independent of $I_3$.
\begin{remark}
\label{remark:weakEmpiricalInterpretation}
	Unlike \eqref{eq:empiricalRedundantConditionHyper}, the inequalities \eqref{eq:hyperWETSS} are accessible to an intuitive geometric interpretation (see\ Figure \ref{fig:WETSS}):
	For a deformation gradient $F\in\GLp(3)$, let
	\begin{equation}
		l(F) = \sqrt{\frac{1}{\abs{S^2}}\,\smash{\int\limits_{\xi\in S^2}}\vphantom{\int} \norm{F\.\xi}^2\,\dS^2}
	\end{equation}
	denote the (Euclidean) average length of unit vectors deformed by $F$, where $S^2\subset\R^3$ is the unit sphere. Then \cite{bavzant1986efficient,lubarda1993damage}
	\begin{equation}
	\label{eq:averageLengthRelatedToTraceB}
		l(F)^2 = \frac{1}{\abs{S^2}}\,\int\limits_{\xi\in S^2} \iprod{B\.\xi,\xi}\,\dS^2 =\frac13\,\tr(B) = \frac13\,I_1\,.
	\end{equation}
	Observe that
	\[
		l(F)^2 = \frac13\.\tr(B) = \frac13\.(\lambda_1^2+\lambda_2^2+\lambda_3^2) \geq (\lambda_1^2\.\lambda_2^2\.\lambda_3^2)^{\frac13} = \det(F)^{\afrac23}
	\]
	due to the inequality of arithmetic and geometric means, where $\lambda_1,\lambda_2,\lambda_3$ are the singular values of $F$, and that equality holds if and only if $\lambda_1=\lambda_2=\lambda_3$ or, equivalently, $B=\det(F)^{\afrac23}\cdot\id$. Thus, for all $d>0$,
	\begin{equation}
	\label{eq:averageLengthConstrainedMin}
		\min_{\substack{F\in\GLp(3)\\\det F = d}} l(F) = l(d^{\afrac13}\cdot\id) = l(d^{\afrac13}\cdot Q)
	\end{equation}
	for any proper rotation $Q\in\SO(3)$.
	
	Under the constraint of a fixed determinant $\det F = d$, it is physically reasonable to assume that a purely volumetric deformation (i.e.\ a conformal mapping of the form $d^{\afrac13}\cdot Q$ with $Q\in\SO(3)$) is \emph{energetically optimal} since it involves no isochoric deformation (pure shape change) of the body in addition to the prescribed change of volume.
	
	If the average deformed length $l(F)$ is considered to be a characteristic quantity of a deformation $F$, it is therefore plausible to assume that it is \emph{energetically preferable} for $l(F)$ to remain close to $l(d^{\afrac13}\cdot Q)$ under the constraint $\det F = d$.
	
	According to \eqref{eq:averageLengthConstrainedMin}, this implies that a smaller value of $l(F)$ should correspond to a lower elastic energy. Since, due to \eqref{eq:averageLengthRelatedToTraceB}, the average length $l(F)$ increases monotonically with $I_1$, these considerations can be expressed by the requirement $\frac{\partial W}{\partial I_1}\geq 0$; note that taking the partial derivative only with respect to the first invariant presupposes that the determinant $\det F$ remains fixed.
	
	Similarly, the inequality $\frac{\partial W}{\partial I_2}\geq 0$ can be motivated by observing that
	\[
		\frac13\,I_2 = \frac13\,\tr(\Cof B) = \frac{1}{\abs{S^2}}\,\int\limits_{\xi\in S^2} \norm{(\Cof F)\.\xi}^2\,\dS
	\]
	measures the average area of unit area plane sections deformed by $F$.
	
	Note carefully that this interpretation cannot be \emph{strictly} applied to the special case $B\in\Rp\cdot\id$ representing the \enquote{energetically ideal} boundary case, which corresponds to a purely volumetric (conformal) deformation $F\in\Rp\cdot\SO(3)$ and is excluded from the requirements in Definition \ref{definition:weakEmpirical}, since $l$ reaches its minimum (for fixed determinant) at these deformations.
\end{remark}
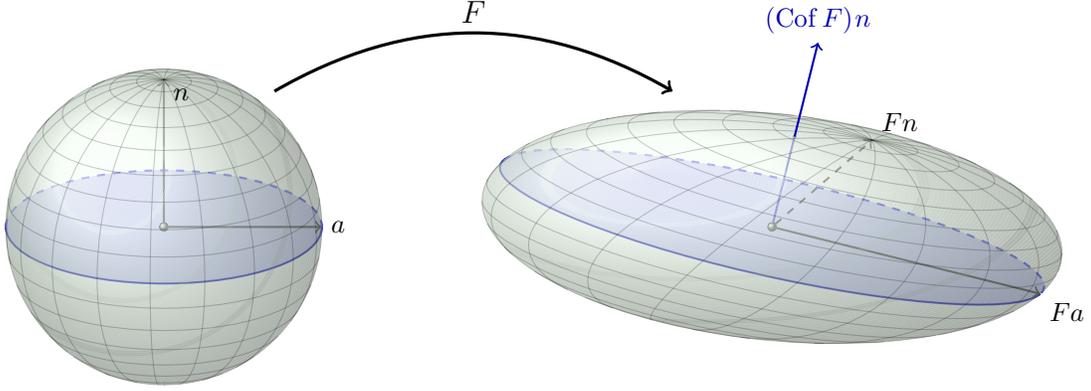
\begin{figure}[h!]
  	\centering
	\begin{tikzpicture}[scale=.84]%
	%
		\def\lone{1.75} 
		\def\ltwo{1/sqrt(\lone)} 
		\def\R{2.5}
		\def\angEl{21}
		\def\angElTwo{24.5}
		\def\shearAmount{.7}
		\def\angRot{14}
		\def\distanceFactor{1.4}
		\newcommand{\gridLines}{7}
		\newcommand{\gridOpacity}{0.35}
		\def\cofFnSplit{0.49}
		\definecolor{sphereColor}{rgb}{.924,.987,.924}
		\newcommand{\sphereOpacity}{.259}
		\newcommand{\circleOpacity}{.21}
		\definecolor{planeColor}{rgb}{0,0,.91}
		\newcommand{\planeOpacity}{.28}
		\colorlet{cofColor}{blue!70!black}
		%
		\def\pointNodeSize{2.1}
		\tikzset{pointNodestyle/.style={inner sep = 0, minimum size=\pointNodeSize, draw, circle, very thin, shade, shading=ball}}
		\makeatletter
		\pgfdeclareradialshading[tikz@ball]{ball}{\pgfqpoint{-10bp}{10bp}}{%
		 color(0bp)=(tikz@ball!10!white);
		 color(9bp)=(tikz@ball!50!white);
		 color(18bp)=(tikz@ball!77!black);
		 color(25bp)=(tikz@ball!60.2!black);
		 color(50bp)=(black)}
		\makeatother
		\newcommand\pgfmathsinandcos[3]{
			\pgfmathsetmacro#1{sin(#3)}
			\pgfmathsetmacro#2{cos(#3)}
		}
		\newcommand\LongitudePlane[3][currentplane]{
		  \pgfmathsinandcos\sinEl\cosEl{#2} 
		  \pgfmathsinandcos\sint\cost{#3}
		  \tikzset{#1/.style={cm={\cost,\sint*\sinEl,0,\cosEl,(0,0)}}}
		}
		\newcommand\LatitudePlane[3][currentplane]{
			\pgfmathsinandcos\sinEl\cosEl{#2}
			\pgfmathsinandcos\sint\cost{#3}
			\pgfmathsetmacro\yshift{\cosEl*\sint}
			\tikzset{#1/.style={cm={\cost,0,0,\cost*\sinEl,(0,\yshift)}}} 
		}
		\newcommand\DrawLongitudeCircle[2][1]{
			\LongitudePlane{\angEl}{#2}
			\tikzset{currentplane/.prefix style={scale=#1}}
			\pgfmathsetmacro\angVis{atan(sin(#2)*cos(\angEl)/sin(\angEl))} 
			\draw[currentplane, color=black, opacity=\circleOpacity] (\angVis:1) arc (\angVis:\angVis+180:1);
		}
		\newcommand\DrawLatitudeCircle[2][1]{
			\LatitudePlane{\angEl}{#2}
			\tikzset{currentplane/.prefix style={scale=#1}}
			\pgfmathsetmacro\sinVis{sin(#2)/cos(#2)*sin(\angEl)/cos(\angEl)}
			\pgfmathsetmacro\angVis{asin(min(1,max(\sinVis,-1)))}
			\draw[currentplane, color=black, opacity=\circleOpacity] (\angVis:1) arc (\angVis:-\angVis-180:1);
		}
		\newcommand\DrawFilledLongitudeCircle[2][1]{
			\LongitudePlane{\angEl}{#2}
			\tikzset{currentplane/.prefix style={scale=#1}}
			\pgfmathsetmacro\angVis{atan(sin(#2)*cos(\angEl)/sin(\angEl))}
			\fill[currentplane,color=blue,opacity=\planeOpacity] (\angVis:1) arc (\angVis:360:1); 
			\draw[currentplane,color=blue,opacity=1, thick] (\angVis:1) arc (\angVis:360:1);
		}
		\newcommand\DrawFilledLatitudeCircle[2][1]{
			\LatitudePlane{\angEl}{#2}
			\tikzset{currentplane/.prefix style={scale=#1}}
			\pgfmathsetmacro\sinVis{sin(#2)/cos(#2)*sin(\angEl)/cos(\angEl)}
			\pgfmathsetmacro\angVis{asin(min(1,max(\sinVis,-1)))}
			\fill[currentplane, color=planeColor, opacity=\planeOpacity] (\angVis:1) arc (\angVis:720:1);
			\draw[currentplane, color=planeColor, opacity=1, thick] (\angVis:1) arc (\angVis:-\angVis-180:1);
			\draw[currentplane, color=planeColor, dashed, opacity=1, thick] (-\angVis+180:1) arc (-\angVis+180:\angVis:1);
		}
		%
		%
		%
		\begin{scope}
			%
			\DrawFilledLatitudeCircle[\R]{0}
			%
			\node[pointNodestyle, ball color = black, minimum size = {1.5*\pointNodeSize}] (centerNode) at (0,0){};
			%
			\draw[thick, ->] (centerNode) -- (\R,0) node (a) {};
			\draw[thick, ->] (centerNode) -- (0,{\R*cos(\angEl)}) node (n) {};
			%
			\filldraw[color = sphereColor, ball color=sphereColor, opacity = \sphereOpacity] (0,0) ellipse ({\R} and \R);
			\foreach \t in {-80,-70,...,80} {\DrawLatitudeCircle[\R]{\t}}
			\foreach \t in {-5,-25,...,-175} {\DrawLongitudeCircle[\R]{\t+10}}
			\draw[color=black, opacity=\circleOpacity] (0,0) ellipse ({\R} and \R);
			%
			\node[right, color=black] at (a){$a$};
			\node[below right, color=black] at (n){$n$};
		\end{scope}
		\begin{scope}[cm={%
		\lone*cos(\angRot),-\lone*sin(\angRot),%
		\ltwo*(sin(\angRot)+\shearAmount*cos(\angRot)),\ltwo*(cos(\angRot)-\shearAmount*sin(\angRot)),%
		({\distanceFactor*(\R+\lone*\R)},0)}]
			%
			\def\angEl{\angElTwo}
			%
			\DrawFilledLatitudeCircle[\R]{0}
			%
			\node[pointNodestyle, ball color = black, minimum size = {1.5*\pointNodeSize}] (centerNode) at (0,0){};
			%
			\draw[thick, ->] (centerNode) -- (\R,0) node (Fa) {};
			\draw[thick, color=cofColor] (centerNode) -- ({\cofFnSplit*(-\shearAmount/sqrt(\lone)*\R*cos(\angEl))},{\cofFnSplit*\lone*\R*cos(\angEl)}) node[inner sep = 0, minimum size=1.75, circle, fill, color=black, opacity=.63] {};
			\draw[->,thick,dashed] (centerNode) -- (0,{\R*cos(\angEl)}) node (Fn) {};
			%
			\filldraw[color = sphereColor, ball color=sphereColor, opacity = \sphereOpacity] (0,0) ellipse ({\R} and \R);
			\foreach \t in {-80,-70,...,80} {\DrawLatitudeCircle[\R]{\t}}
			\foreach \t in {-5,-25,...,-175} {\DrawLongitudeCircle[\R]{\t+10}}
			\draw[color=black, opacity=\circleOpacity] (0,0) ellipse ({\R} and \R);
			%
			\draw[thick, ->, color=cofColor] ({\cofFnSplit*(-\shearAmount/sqrt(\lone)*\R*cos(\angEl))},{\cofFnSplit*\lone*\R*cos(\angEl)}) -- ({-\shearAmount/sqrt(\lone)*\R*cos(\angEl)},{\lone*\R*cos(\angEl)}) node (CofFn) {};
			%
			\node[below right, color=black] at (Fa) {$Fa$};
			\node[above, color=cofColor] at (CofFn) {$(\Cof F)\.n$};
			\node[above right, color=black] at (Fn) {$Fn$};
		\end{scope}
		%
		\draw[->,very thick] (1.75,{\R-.35}) to[bend left] node[above] {\large $F$} (8.05,{\R-.35});
	\end{tikzpicture}%
	\caption{The change in length and area induced by the deformation gradient $F$ are given by $\norm{Fa}$ and $\norm{(\Cof F)\.n}$, respectively, where $n$ is a unit normal to the deformed plane.}\label{fig:WETSS}
\end{figure}
The condition \eqref{eq:empiricalRedundantConditionHyper} which distinguishes the weak empirical inequalities from the classical ones, in addition to being redundant for many applications, also cannot be satisfied for certain (physically reasonable) classes of  energy potentials. Consider the purely volumetric homogeneous deformation $\varphi(x)=\sqrt\lambda\.x$ with $\lambda\geq 1$. Of course, it is a reasonable assumption that the energy function $W(F)$ should be monotone with respect to $\lambda$, i.e.\ $\ddd{W}{\lambda}>0$. For $F=\grad\varphi=\sqrt{\lambda}\.\id$, the principal invariants of $B=FF^T$ are given by
\begin{align}
	I_1=\tr(B)=3\.\lambda\,,\quad I_2=\tr(\Cof B)=3\.\lambda^2\,,\quad I_3=\det(B)=\lambda^3\,.\label{eq:inv}
\end{align}
Expressing the energy $W$ by
\[
	W(I_1,I_2,I_3)=\widetilde W(J_1,J_2,J_3)=\widetilde W(I_1I_3^{-\frac{1}{3}},I_2I_3^{-\frac{2}{3}},I_3)\,,
\]
the condition $\ddd{W}{\lambda}>0$ can be stated as
\[
	0<\ddd{W}{\lambda}W(3\.\lambda,3\.\lambda^2,\lambda^3)=\ddd{W}{\lambda}\widetilde W(3,3,\lambda^3)=\frac{\partial\widetilde W}{\partial J_3}\.\underbrace{3\.\lambda^2}_{>0}\qquad\text{for all }\;\lambda>1\,.
\]
Truesdell's condition \eqref{eq:empiricalRedundantConditionHyper} on $\beta_0$ implies
\[
	0\geq\frac{\partial W}{\partial I_3}=\dd{I_3}\widetilde W(I_1I_3^{-\frac{1}{3}},I_2I_3^{-\frac{2}{3}},I_3)=-\frac{1}{3}\.I_1\.\frac{\partial \widetilde W}{\partial J_1}\.I_3^{-\frac{4}{3}}-\frac{2}{3}\.I_2\.\frac{\partial \widetilde W}{\partial J_2}\.I_3^{-\frac{5}{3}}+\frac{\partial\widetilde W}{\partial J_3}\,.
\]
For the above volumetric deformation, in particular, \eqref{eq:empiricalRedundantConditionHyper} implies
\begin{align}
	J_3\.\frac{\partial\widetilde W}{\partial J_3}\leq\frac{\partial \widetilde W}{\partial J_1}+2\.\frac{\partial \widetilde W}{\partial J_2}\qquad\text{for all}\quad\lambda>1\quad\text{with}\quad J_1=J_2=3\,,\quad J_3=\lambda^3\,.\label{eq:beta0Volumetric}
\end{align}
However, most elastic energy functions commonly considered in applications do not satisfy this condition. For example, consider the class of energy functions with an additive \emph{isochoric-volumetric split}, i.e.\ an energy $W$ the form
\[
	W(I_1,I_2,I_3)=\Wiso(I_1I_3^{-\frac{1}{3}},I_2I_3^{-\frac{2}{3}})+f(I_3)
\]
with a scalar-valued function $f:\Rp\to\R$ that satisfies $f'(1)=0$ to ensure a stress-free reference configuration. Then \eqref{eq:beta0Volumetric} would imply
\[
	\lambda^3\.f'(\lambda^3)
	\leq\frac{\partial}{\partial J_1}\Wiso(3,3)+2\.\frac{\partial}{\partial J_2}\Wiso(3,3)=\textit{const.}\qquad\text{for all }\;\lambda>1
\]
which, in turn, implies that $f'$ and hence the hydrostatic pressure $\tr\sigma$ is bounded above for large volumetric strains. Condition \eqref{eq:empiricalRedundantConditionHyper} thereby excludes isochoric-volumetrically split energy functions which exhibit physically plausible behavior.

%
%
%
%
\subsection{The quadratic Hencky energy}
\label{section:henckyExample}
As an example for the advantages offered by the weak empirical inequalities compared to Truesdell's (full) emprical inequalities, we consider the classical \emph{quadratic-logarithmic Hencky energy} $\WH\colon\GLp(3)\to\R$ with \cite{Hencky1928,Hencky1929,agn_neff2013hencky,agn_neff2015geometry}
\begin{equation}
	\WH(F) = \mu\,\norm{\dev_3 \log U}^2 + \frac \kappa 2\left(\tr(\log U)\right)^2 = \mu\,\norm{\log U}^2 + \frac \Lambda 2\left(\tr(\log U)\right)^2\,,
\end{equation}
where $U=\sqrt{F^TF}$ denotes the right Biot stretch tensor, $\dev_3 X=X-\frac{1}{3}\.\tr(X)$ is the deviatoric part of $X\in\R^{3\times3}$, $\mu>0$ is the shear modulus, $\kappa>0$ is the bulk modulus and $\Lambda$ is the first Lam\'e parameter with $3\Lambda+2\mu\geq 0$. While it is well known that the Cauchy stress response induced by $\WH$ does not satisfy $\etss$, we will show here that it does fulfil $\wetss$. Our proof is based on the so-called \emph{sum of squared logarithms inequality} (SSLI) \cite{birsan2013sum,agn_lankeit2014minimization}.
\begin{proposition}\label{def:ssli}
	Let $a_1$, $a_2$, $a_3$, $b_1$, $b_2$, $b_3\in\R^+$ such that
	\begin{align}
		a_1+a_2+a_3\leq b_1+b_2+b_3\,,\qquad a_1\.a_2+a_1\.a_3+a_2\.a_3\leq b_1\.b_2+\.b_1\.b_3+b_2\.b_3\,,\qquad a_1\.a_2\.a_3=b_1\.b_2\.b_3\,.
	\end{align}
	Then
	\begin{align}
	\label{eq:ssliImpliedInequality}
		(\log a_1)^2+(\log a_2)^2+(\log a_3)^2\leq (\log b_1)^2+(\log b_2)^2+(\log b_3)^2\,,
	\end{align}
	and strict inequality holds if $(a_1,a_2,a_3)\neq(b_1,b_2,b_3)$.
\end{proposition}
With $a_i$ and $b_i$ as the eigenvalues of two symmetric matrices $X,\widetilde X\in\PSym(3)$, the SSLI can be stated in terms of the matrix invariants as
\begin{align}
	\left.\begin{matrix}
		I_1(X)\leq I_1(\widetilde X)\\I_2(X)\leq I_2(\widetilde X)\\I_3(X)= I_3(\widetilde X)
	\end{matrix}\right\}\qquad\implies \qquad \norm{\log X}^2\leq \norm{\log\widetilde X}^2\,.
\end{align}
The SSLI thereby can be interpreted as the monotonicity of the logarithmic strain measure $X\mapsto\norm{\log X}^2$ with respect to the first two matrix invariants, which in turn is directly related to the weak empirical inequalities.
\begin{theorem}\label{theo:logUSemi}
	The Cauchy stress response induced by the energy function $W\col\GLp(3)\to\R$ with
	\begin{align}
		W(F)=\norm{\log U}^2\qquad\text {for all }\; F\in\GLp(3)\,,
	\end{align}
	where $U=\sqrt{F^TF}\Symp(3)$ denotes the right Biot stretch tensor, satisfies the weak empirical inequalities.
\end{theorem}
\begin{proof}
	Let $\lambda_1$, $\lambda_2$, $\lambda_3\in\R^+$ denote the singular values of the deformation gradient $F$.
	
	The SSLI ensures the monotonicity of $\norm{\log B}^2$ with respect to the matrix invariants $I_1(B)$ and $I_2(B)$ or, equivalently, the monotonicity of $\norm{\log U}^2$ due to the equality
	\[
		\norm{\log B}^2=\sum_{i=1}^3\left[\log\left(\lambda_i^2\right)\right]^2=\sum_{i=1}^3\left[2\.\log(\lambda_i)\right]^2=4\.\sum_{i=1}^3\left[\log(\lambda_i)\right]^2=4\.\norm{\log U}^2\,.
	\]
	Therefore,\footnote{Note that, since the invariant representation of the quadratic Hencky energy is not easily accessible, it would be difficult to obtain the same result based on a direct calculation of the derivatives.}
	\begin{align}
		\frac{\partial W}{\partial I_1}(I_1,I_2,I_3)\geq 0\qquad\text{and}\qquad\frac{\partial W}{\partial I_2}(I_1,I_2,I_3)\geq 0\,.\label{help:2}
	\end{align}
	It remains to show that one of the above inequalities \eqref{help:2} is strict. The SSLI implies strictness of both inequalities in the case of distinct eigenvalues, so without loss of generality, let $\lambda_1=\lambda_2\colonequals a$ and $\lambda_3\colonequals b$ with $a\neq b$ and assume that
	\begin{equation}\label{eq:henckySingularToContradict}
		\frac{\partial W}{\partial I_1}(I_1,I_2,I_3)=\frac{\partial W}{\partial I_2}(I_1,I_2,I_3)=0\,.
	\end{equation}
	Then for
	\[
		B\col\R\to\PSym(3)\,,\quad B(t)=\diag\left(a\.(1+t), a , \frac{b}{1+t}\right) \qquad\text{and}\qquad U(t)=\sqrt{B(t)}\,,
	\]
	we find
	\[
		I_1(B(t)) = a\.(1+t)+ a+\frac{b}{1+t}\,,\qquad I_2(B(t)) = a^2\.(1+t)+ab\.\frac{2+t}{1+t}\,,\qquad I_3(B(t))=a^2b
	\]
	as well as
	\begin{alignat*}{2}
		&W(U(t)) &= \frac14\.\norm{\log B(t)}^2&= \frac14\.\left(\log \left( a\,(1+t)\right)^2 +\left( \log a\right)^2 + \log\left( \frac b{1+t}\right)^2\right)\\
		&&&= \frac14\.\left( \left(\log a + \log (1+t)\right)^2 + (\log a)^2 + \left(\log b - \log (1+t)\right)^2 \right)\\
		&&&= \frac12\,(\log a)^2 + \frac14(\log b)^2 + \frac12(\log a - \log b)\log(1+t) + \frac12\left(\log (1+t)\right)^2\,,\\[.7em]
		\ddt &\mathrlap{W(B(t)) = \frac12\.(\log a - \log b)\.\frac 1{1+t} + \log(1+t)\.\frac 1{1+t}\,.}&&
	\end{alignat*}
	But then, since $\ddt I_3(B(t))=0$,
	\[
		0 \neq \frac12\.(\log a - \log b) = \ddt W(U(t))\big|_{t=0} = \frac{\partial W}{\partial I_1}(B(0))\cdot (a-b) + \frac{\partial W}{\partial I_2}(B(0))\cdot (a^2-ab)\,,
	\]
	contradicting \eqref{eq:henckySingularToContradict}.
\end{proof}
The results of Theorem \ref{theo:logUSemi} can be extended to more general classes of hyperelastic stress responses. First, we consider energy functions \emph{in terms of} the classical logarithmic strain $\norm{\log U}^2$, which exhibit a number of interesting properties \cite{agn_neff2015exponentiatedI}.
\begin{corollary}
Let $f\colon\R_+\to\R$ be differentiable with $f'>0$. Then the Cauchy stress response induced by the energy function
\begin{equation}
	W\colon\GLp(3)\to\R\,,\quad W(U) = f\bigl(\norm{\log U}^2\bigr)\,,
\end{equation}
where $U=\sqrt{F^TF}\in\Symp(3)$ denotes the right Biot stretch tensor, satisfies the weak empirical inequalities.
\end{corollary}
\begin{proof}
For $k\in\{1,2\}$ and $U\in\PSym(3)\setminus \R\cdot\id$,
\[
	\dd{I_k}f\bigl(\norm{\log U}^2\bigr) = \dd{I_k}f\Bigl(\frac 14\.\norm{\log B}^2\Bigr) = \frac 14\,\underbrace{f'\Bigl(\frac 14\norm{\log B}^2\Bigr)}_{>0}\,\cdot\,\dd{I_k}\.\norm{\log B}^2 \geq 0\,,
\]
where equality cannot hold for both $k=1$ and $k=2$ as shown in the proof of Theorem \ref{theo:logUSemi}.
\end{proof}
Theorem \ref{theo:logUSemi} can be further generalized to so-called \emph{Hencky-type} energy functions of the form
\begin{equation}\label{eq:henckyType}
	W(F) = \mathcal W\bigl(\norm{\dev_3\log U}^2,\ \dynabs{\tr\log U}^2\bigr)\,.
\end{equation}
\begin{theorem}\label{theo:henckyTypeSemi}
	Let $W\colon \GLp(3)\to\R$ be an elastic energy function of the form \eqref{eq:henckyType} with positive partial derivative with respect to $\norm{\dev_3\log U}^2$, i.e.
	\begin{equation}
		\frac{\partial \mathcal W}{\partial x_1}\bigl(\norm{\dev_3\log U}^2,\ \dynabs{\tr\log U}^2\bigr) > 0\qquad\text{for all }\;U\in\PSym(3)\,.
	\end{equation}
	Then the Cauchy stress response induced by $W$ satisfies the weak empirical inequalities.
\end{theorem}
\begin{proof}
	Recall that $\norm{\log U} = \frac 12\norm{\log B}$ and $\norm{\log B}^2=\norm{\dev_n\log B}^2 + \frac 1n\abs{\tr\log B}^2$. We compute
	\begin{align*}
		\frac{\mathrm d\mathcal W}{\mathrm dI_k}\bigl(\norm{\dev_3\log U}^2,\ \dynabs{\tr\log U}^2\bigr)&=\frac{\partial\mathcal W}{\partial x_1}\bigl(\norm{\dev_3\log U}^2,\ \dynabs{\tr\log U}^2\bigr)\cdot\frac{\mathrm d}{\mathrm dI_k}\norm{\dev_3\log U}^2\\
		&\phantom{=}\;+\frac{\partial\mathcal W}{\partial x_2}\bigl(\norm{\dev_3\log U}^2,\ \dynabs{\tr\log U}^2\bigr)\cdot\frac{\mathrm d}{\mathrm dI_k}\abs{\tr\log U}^2\\[.5em]
		&=\frac 14\.\frac{\partial\mathcal W}{\partial x_1}\bigl(\norm{\dev_3\log U}^2,\ \dynabs{\tr\log U}^2\bigr)\cdot\frac{\mathrm d}{\mathrm dI_k}\bigl(\norm{\log B}^2-\frac 13\abs{\tr\log B}^2\bigr)\\
		&\phantom{=}\;+ \frac 14\.\frac{\partial\mathcal W}{\partial x_2}\bigl(\norm{\dev_3\log U}^2,\ \dynabs{\tr\log U}^2\bigr)\cdot\frac{\mathrm d}{\mathrm dI_k}\abs{\tr\log B}^2\\[.5em]
		&=\frac 14\cdot\frac{\partial\mathcal W}{\partial x_1}\bigl(\norm{\dev_3\log U}^2,\ \dynabs{\tr\log U}^2\bigr)\cdot\frac{\mathrm d}{\mathrm dI_k}\norm{\log B}^2\\
		&\phantom{=}\;+\frac 14\.\left(-\frac 13\.\frac{\partial\mathcal W}{\partial x_1}+\frac{\partial\mathcal W}{\partial x_2}\right)\bigl(\norm{\dev_3\log U}^2,\ \dynabs{\tr\log U}^2\bigr)\cdot\frac{\mathrm d}{\mathrm dI_k}\abs{\tr\log B}^2\,.
	\end{align*}
	 For $k\in\{1,2\}$, we find $\frac{\mathrm d}{\mathrm dI_k}\abs{\tr\log B}^2=\frac{\mathrm d}{\mathrm dI_k}\abs{\log I_3}^2=0$. Therefore,
	\[
		\frac{\mathrm d\mathcal W}{\mathrm dI_k}(U)\ =\ \underbrace{\frac{\partial\mathcal W}{\partial x_1}\bigl(\norm{\dev_3\log U}^2,\ \dynabs{\tr\log U}^2\bigr)}_{>0}\,\cdot\,\frac{\mathrm d}{\mathrm dI_k}\norm{\log B}^2\ \geq\ 0\qquad\text{for }\;k\in\{1,2\}\,;
	\]
	as in the proof of Theorem \ref{theo:logUSemi}, the SSLI shows that $\frac{\mathrm d}{\mathrm dI_k}\norm{\log B}^2 \geq 0$ for $k\in\{1,2\}$ with at least one of the two inequalities being strict, which concludes the proof.
\end{proof}
The most well-known example of Hencky-type energy functions is the classical \emph{quadratic Hencky energy} \cite{Hencky1928,Hencky1929,agn_neff2013hencky,agn_neff2015geometry} $\WH\colon\GLp(3)\to\R$ with
\begin{equation}
	\WH(F)\ =\ \mu\,\norm{\dev_3 \log U}^2 + \frac \kappa 2\left(\tr(\log U)\right)^2\ =\ \mu\,\norm{\log U}^2 + \frac \Lambda 2\left(\tr(\log U)\right)^2\,,
\end{equation}
where $U=\sqrt{F^TF}$. However, a number of alternative energy potentials based on logarithmic strain measures have been suggested as a model for large elastic deformations, including the recently introduced \emph{exponential Hencky energy} \cite{agn_neff2015exponentiatedI,agn_neff2015exponentiatedII,nedjar2018finite} $\WeH\colon\GLp(3)\to\R$ with
\begin{equation}
	\WeH(F)\ =\ \frac \mu k\,e^{k\,\norm{\dev_3\log U}^2} + \frac \kappa{2\widehat k}\,e^{\widehat k\,\abs{\tr \log U}^2}\,,\qquad \mu,k,\kappa,\widehat k\in\R_+\,.
\end{equation}
Due to Theorem \ref{theo:henckyTypeSemi}, the elastic laws induced by both $\WH$ and $\WeH$ satisfy the weak empirical inequalities and are therefore semi-invertible. Unlike $\WH$, however, it can be shown that $\WeH$ even induces an invertible (in the classical sense) Cauchy stress response \cite{agn_neff2015exponentiatedI}.

For additional examples of hyperelastic energy functions which satisfy $\wetss$ but not $\etss$, see \cite[p.138]{agn_thiel2017neue}.

\newpage
%
%
%
\footnotesize
\section{References}
\printbibliography[heading=none]
%
%
%
%
\begin{appendix}
%
%
%
%
\section{Uniaxial tension and the Baker-Ericksen inequalities}\label{appendix:marzano}
In an 1983 article, Marzano \cite{marzano1983interpretation} showed that for every constitutive law of elasticity satisfying the Baker-Ericksen inequalities, a uniaxial tension Cauchy stress tensor of the form $\sigma=\diag(s,0,0)$ can only be caused by a simple stretch of the form $V-\id=\diag(\alpha,0,0)$, where $V=\sqrt{FF^T}$ denotes the left Biot stretch tensor. Marzano also remarked that \enquote{an isotropic elastic material satisfies the B-E inequalities if and only if a simple tension in the direction of $e_3$ produces a simple extension in the same direction, with the ratio of the lateral contraction to the longitudinal strain positive and less than 1} \cite[p.~234]{marzano1983interpretation}.

However, this statement, interpreted literally, is not true in general: While Marzano indeed shows that the implication $\lambda_i>\lambda_j \implies \sigma_i>\sigma_j$ holds \emph{if $\lambda_1=1+\alpha$ and $\lambda_2=\lambda_3=0$}, he does not establish the Baker-Ericksen inequalities \enquote{globally}, i.e.\ for all $\lambda_1,\lambda_2,\lambda_3\in\Rp$. As a counterexample, consider the mapping
\[
	V\mapsto\sigma(V)\ =\ \bigl(1-h(\lambda_1,\lambda_2,\lambda_3)\bigr)\. V - \id
\]
with the symmetric function $h\colon\R_+^3\to\R$ depending on the eigenvalues of $V$ with
\[
	h(\lambda_1,\lambda_2,\lambda_3)\ =\ (\lambda_1-\lambda_2)^2(\lambda_1-\lambda_3)^2(\lambda_2-\lambda_3)^2\,,
\]
which is zero if and only if two eigenvalues are equal. Then
\[
\sigma(V)\ =\ \matr{s&0&0\\0&0&0\\0&0&0}\qquad\iff\qquad \begin{cases}\bigl(1-h(\lambda_1,\lambda_2,\lambda_3)\bigr)\,\lambda_1-1&=\ s\,,\\[.5em]\bigl(1-h(\lambda_1,\lambda_2,\lambda_3)\bigr)\,\lambda_2-1&=\ 0\,,\\[.5em]\bigl(1-h(\lambda_1,\lambda_2,\lambda_3)\bigr)\,\lambda_3-1&=\ 0\,,\end{cases}
\]
which implies that $\lambda_2=\lambda_3$ and thus $h(\lambda_1,\lambda_2,\lambda_3)=0$, i.e.\ $\lambda_1=1+s$, $\lambda_2=\lambda_3=1$ and hence $V=\diag(1+s,1,1)$. Therefore, uniaxial tension can only be induced by a simple stretch. However, $\sigma$ does not satisfy the Baker-Ericksen inequalites; for example, if $V=\diag(3,2,1)$, then $h(3,2,1)=2$ and thus $\sigma\bigl(\diag(3,2,1)\bigr)=-\diag(3,2,1)-\id=\diag(-4,-3,-2)$. In this case, $\lambda_1=3>\lambda_2=2$ and $\sigma_1=-4<\sigma_2=-3$, contradicting (BE).
\end{appendix}
\end{document}